\documentclass[ngerman,twoside,numbers=noenddot]{scrartcl}
\KOMAoptions{DIV=12, BCOR=0mm, paper=a4, fontsize=12pt, parskip=half, twoside=true, titlepage=true}
\usepackage[onehalfspacing]{setspace} 
\usepackage{authblk}
\usepackage[headsepline,automark,pagestyleset=KOMA-Script,markcase=ignoreuppercase]{scrlayer-scrpage}
\clearmainofpairofpagestyles
\ohead{\pagemark}
\ihead{\headmark}
\usepackage[left=3cm,right=3cm,top=1.5cm,bottom=3cm,head=35.9889pt,includehead]{geometry}
\usepackage[utf8]{inputenc} 
\usepackage[english]{babel} 
\usepackage[expansion=true, protrusion=true]{microtype}
\usepackage{amsmath}
\numberwithin{equation}{section}
\usepackage{calc}
\usepackage{mathtools}
\usepackage{amsfonts}
\usepackage{amssymb} 
\usepackage{amsthm}
\usepackage{upgreek}
\usepackage{bbold}
\usepackage{enumitem}
\usepackage[dvipsnames]{xcolor}
\usepackage{units}
\usepackage{graphicx} 
\usepackage[hypcap]{caption} 
\usepackage{booktabs} 
\usepackage{flafter} 
\usepackage[section]{placeins}
\usepackage{tikz-cd}
\usepackage{slashed}
\usepackage{hyperref} 
\hypersetup{colorlinks=true, breaklinks=true, citecolor=red, linkcolor=blue, menucolor=red, urlcolor=cyan, bookmarksopen=false, bookmarksopenlevel=0, plainpages=false,hypertexnames=false}
\newtheorem{theorem}{Theorem}[section]
\newtheorem{proposition}[theorem]{Proposition}
\newtheorem{lemma}[theorem]{Lemma}
\newtheorem{corollary}[theorem]{Corollary}
\newtheorem{remark}[theorem]{Remark}

\theoremstyle{definition}

\usepackage[T1]{fontenc}
\usepackage{lmodern}
\usepackage{pdfpages}

\usepackage[autostyle]{csquotes}
\usepackage[numbers]{natbib}

\newcommand\restr[2]{{
  \left.
  #1 
  \right|_{#2} 
  }}
\makeatletter
\newcommand*\owedge{\mathpalette\@owedge\relax}
\newcommand*\@owedge[1]{%
  \mathbin{%
    \ooalign{%
      $#1\m@th\bigcirc$\cr
      \hidewidth$#1\m@th\wedge$\hidewidth\cr
    }%
  }%
}
\makeatother

\newcommand{\divergence}{\mathrm{div}}

\newcommand{\extd}[0]{\mathrm{d}}

\newcommand{\riem}{\mathrm{Rm}}
\newcommand{\ric}{\mathrm{Rc}}

\newcommand{\rscal}{\mathrm{Sc}}

\newcommand{\genmet}{\mathcal{G}}

\newcommand{\genscal}{\mathcal{S}c}
\newcommand{\genfullric}{\overline{\mathcal{R}c}}
\newcommand{\genric}{\mathcal{R}c}
\newcommand{\genriem}{\mathcal{R}m}

\newcommand{\scalbrack}[1]{\left\langle #1 \right\rangle}
\newcommand{\scalprod}[2]{\scalbrack{#1, #2}}
\newcommand{\scalprodmap}{\scalprod{\cdot}{\cdot}}

\newcommand{\End}{\mathrm{End}}

\newcommand{\Sym}{\mathrm{Sym}}

\newcommand{\absolute}[1]{\left\lvert #1 \right\rvert}

\newcommand{\trwith}[1]{\mathrm{tr}_{#1}\:}

\newcommand{\sym}{\mathrm{sym}}
\newcommand{\antisym}{\mathrm{antisym}}

\begin{document}
\pagenumbering{Roman}
\thispagestyle{empty}
\cleardoublepage
\setcounter{page}{1}
\thispagestyle{empty}

\cleardoublepage
\pagenumbering{arabic}

\title{The canonical generalised Levi-Civita connection and its curvature}

\date{\today}

\author{Vicente Cort\'es, Matas Mackevicius, Thomas Mohaupt, and Oskar Schiller}

\maketitle

\tableofcontents
\section*{Abstract}
Given a (semi-Riemannian) generalised metric $\mathcal G$ and a divergence operator $\mathrm{div}$ on an exact Courant algebroid $E$, we geometrically construct a canonical generalised Levi-Civita connection $D^{\mathcal G, \mathrm{div}}$ for these data.   In this way we provide a resolution of the problem of non-uniqueness 
of generalised Levi-Civita connections. Since the generalised Riemann tensor of  $D^{\mathcal G, \mathrm{div}}$ is 
an invariant of the pair $(\mathcal G, \mathrm{div})$,  we no longer need to 
discard curvature components which depend on the 
choice of the generalised connection.  As a main result we decompose   
the generalised Riemann curvature tensor of $D^{\mathcal G, \mathrm{div}}$ in terms of 
classical (non-generalised) geometric data. Based on this set of master formulas we derive
a comprehensive curvature tool-kit for applications in generalised geometry. 
This includes decompositions for the full generalised Ricci tensor, 
the generalised Ricci tensor, and three generalised scalar-valued curvature invariants, two of which are new.  

\textit{MSc classification}: 53D18 (Generalized geometry a la Hitchin); 83C10 (Equations of motion in general relativity and gravitational theory).\

\textit{Key words:} generalised metrics, generalised curvature, generalised Einstein equations, supergravity. 
\section{Introduction}
This article is a contribution to the field of generalised geometry understood as the 
study of geometric structures on Courant algebroids. We are specifically concerned 
with generalised (semi-Riemannian) metrics $\mathcal G$ on exact Courant algebroids $(E, \pi, \langle \cdot , \cdot \rangle, [\cdot , \cdot ])$ over a $d$-dimensional manifold $M$ and their generalised Levi-Civita connections, that is torsion-free metric generalised connections. 

Our first result is a canonical geometric construction 
of a generalised Levi-Civita connection $D^{\mathcal G, \mathrm{div}}$ for any prescribed divergence operator $\mathrm{div}$.  The basic idea of our construction
is that the Levi-Civita connection of the underlying semi-Riemannian metric $g$ (defined in Section \ref{LC:sec}) has a natural extension 
to a generalised Levi-Civita connection $D^0$ with metric divergence 
$\mathrm{div}^{\mathcal G}=\mathrm{div}^g \circ \pi$. In terms of $D^0$ the canonical generalised Levi-Civita connection $D^{\mathcal G, \mathrm{div}}$ is simply characterised by its divergence operator $\mathrm{div}$ and the property that the tensor $D^{\mathcal G, \mathrm{div}}-D^0$
is (point-wise) perpendicular to the kernel of the trace map \eqref{trace:eq}.  Note that 
$D^0=D^{\mathcal G, \mathrm{div}^{\mathcal G}}$.  As an obvious but important consequence, any invariant of $D^{\mathcal G, \mathrm{div}}$
is an invariant of the pair $(\mathcal G, \mathrm{div})$. 

By comparing the resulting formulas in terms of a splitting of the Courant algebroid, we show that 
$D^{\mathcal G, \mathrm{div}}$ coincides with a generalised 
Levi-Civita connection that appeared 
in the work of Garcia-Fernandez and Streets \cite{genricciflow}. The latter was obtained from the Gualtieri-Bismut generalised connection  and involves the classical Bismut connections $\nabla^\pm$ associated with the underlying three-form $H$ (and another pair of connections $\nabla^{\pm 1/3}$). Our approach has the advantage that the generalised Levi-Civita connection $D^{\mathcal G, \mathrm{div}}$ is canonical by construction. Moreover, it provides a new conceptual perspective on, and a fast track to, the final formulas of \cite{genricciflow}. 

The main results of this paper are the general formulas for the curvatures of the canonical generalised Levi-Civita connection $D^{\mathcal G, \mathrm{div}}$. These provide a powerful and comprehensive curvature tool-kit for practitioners of generalised geometry. Our formulas express the generalised Riemann curvature of $D^{\mathcal G, \mathrm{div}}$ in terms of the following data: 
\begin{enumerate}
    \item 
the semi-Riemannian metric $g$ and its Riemann tensor $\riem$,  
\item the fields $H$,  
$e=2(X+\xi )\in \Gamma (\mathbb{T}M)$, and  
their covariant derivatives, \end{enumerate}where $\langle e, \cdot \rangle =\mathrm{div}^\mathcal{G}-\mathrm{div}$. Here we have identified the Courant algebroid $E$ with the generalised 
tangent bundle $\mathbb{T}M=TM \oplus T^*M$ by means of the generalised metric $\mathcal G$. 
The basic theorem for the generalised Riemann tensor, from which 
the remaining curvature formulas are derived, is the following summary of Theorems~\ref{puretyperiemprop} and \ref{mixedtyperiemprop}. To state it recall that the generalised metric 
defines a decomposition $E=E_+\oplus E_-$ into eigenbundles of $\mathcal{G}^{\mathrm{End}} = \eta^{-1}\mathcal G$, 
where $\eta^{-1}: E^* \to E$ is the inverse of the scalar product $\eta = \langle \cdot , \cdot  \rangle$.

\begin{theorem} The generalised Riemann curvature $\genriem^{D}$ of the canonical generalised Levi-Civita connection 
$D=D^{\mathcal G, \mathrm{div}}$ is a sum of the following two tensors in 
$\Sym^2\,\Lambda^2\,E^*_+\oplus \Sym^2\,\Lambda^2\,E^*_-$ (pure-type) and $(E_+^*\wedge E_-^*)\vee \Lambda^2\,E^*_+ \oplus 
    (E_-^*\wedge E_+^*)\vee \Lambda^2\,E^*_-$ (mixed-type).
The pure-type component of $\genriem^{D}$ is  given by
    \begin{equation*}
        \begin{split}
            &\pm\genriem^{D}(a,b,v,w) \\
            &= \riem(a,b,v,w)  - \frac{1}{36} H^{(2)}(a,v,b,w) - \frac{1}{36}H^{(2)}(b,v,w,a) - \frac{1}{18} H^{(2)}(v,w,a,b) \\
            &\quad \pm \frac{1}{2(d-1)} \left\{ [D^0_v \chi_\pm^{e_\pm}](w,b,a) - [D^0_w \chi_\pm^{e_\pm}](v,b,a)\right. \\
            &\qquad\qquad\qquad \left. + [D^0_b \chi_\pm^{e_\pm}](a,v,w) - [D^0_a \chi_\pm^{e_\pm}](b,v,w) \right\} \\
            &\quad + \frac{1}{2(d-1)^2} \left\{2 \genmet(e_\pm,e_\pm) \big[\genmet(w,a)\genmet(v,b)-\genmet(v,a)\genmet(w,b)\big] \right. \\
            &\qquad\qquad\qquad + \genmet(a,e_\pm) \big[\genmet(w,b)\genmet(v,e_\pm)-\genmet(v,b)\genmet(w,e_\pm)\big] \\
            &\qquad\qquad\qquad + \left.\genmet(b,e_\pm) \big[\genmet(v,a)\genmet(w,e_\pm)-\genmet(w,a)\genmet(v,e_\pm)\big]\right\},
        \end{split}
    \end{equation*}
    where $a,b,v,w \in \Gamma(E_\pm)$ are all of the same type, $D^0$ is the canonical generalised Levi-Civita connection with metric divergence and the remaining notation is explained in Section~\ref{metric_divergence_section}. 
    
        The mixed-type component of $\genriem^{D}$ is given by
    \begin{equation*}
        \begin{split}
            &\pm2 \genriem^D(a,\Bar{b},v,w) \\
            &= \riem(a,\Bar{b},v,w) \mp \frac{1}{2}[\nabla_{a} H](\Bar{b},v,w) \pm \frac{1}{6}[\nabla_{\Bar{b}} H](a,v,w) \\
            &\quad - \frac{1}{12} H^{(2)}(\Bar{b},w,a,v) - \frac{1}{12} H^{(2)}(w,a,\Bar{b},v) - \frac{1}{6} H^{(2)}(a,\Bar{b},v,w) \\
            &\quad \pm \frac{1}{d-1}  [D^0_{\Bar{b}} \chi_\pm^{e_\pm}](a,v,w).            
        \end{split}
    \end{equation*}
    Herein, $a,v,w \in \Gamma(E_\pm)$ are all of the same type and $\Bar{b}\in \Gamma(E_\mp)$.
\end{theorem} 

Tracing the above formulas we obtain corresponding expressions for the full generalised 
Ricci tensor (see Corollaries \ref{genric_puretype_cor} and \ref{genriccor}) and two scalar curvatures (Corollaries~\ref{traG:cor} and \ref{trE:cor}). The formulas for the pure-type components of the  generalised Riemann and full Ricci tensor are new and the same holds for the mixed components of the  generalised Riemann in the case of general divergence operator. In addition, we define a generalised version of the Kretschmann 
scalar, see Proposition \ref{Kretsch:prop}.  The second generalised scalar curvature and the generalised Kretschmann 
scalar are new scalar-valued invariants. 

We expect that the general curvature formulas obtained in this paper will serve the community as  
fundamental tools for the study of any problem in mathematics or  physics 
involving the generalised Riemann tensor. First applications were already 
given in \cite{vicenteoskar} and include the derivation of the energy and momentum 
constraints for the generalised Einstein equations and the characterisation of 
exact Courant algebroids with generalised metrics and divergence operators for which the canonical Levi-Civita connection is flat. Another related application described there is the extension of the fundamental theorem of hypersurface theory to the setting of Hitchin's generalised geometry.

Our results will be useful for future applications of generalised geometry in theories of gravitation, including string theory and supergravity. 
T-duality, which when acting with compact orbits is a symmetry of string theory, 
can be used more broadly as a solution-generating technique in gravitational theories \cite{Buscher:1987sk,Giveon:1994fu}.
Generalised geometry allows one to describe the action of T-duality in a manifestly
covariant way for exact and more general Courant algebroids \cite{Cavalcanti-Gualtieri,Baraglia-Hekmati,Cortes-David_Tduality}. 
One concrete future application of our results is the generalised geometry of
non-extreme Killing horizons and of their T-duals, following 
\cite{Medevielle:2023jmn}. 
The Riemann tensor of the canonical generalised Levi-Civita connnection is the 
basic building block for the curvature invariants appearing in the actions and field 
equations of higher derivative gravitational theories which exhibit T-duality. Such 
theories may be considered as fundamental or effective theories (derived from 
theories of quantum gravity such as string theory). We refer to 
\cite{Ozkan:2024euj} for a recent review of higher-derivative gravity with detailed references. 
A possible future application of our work is to develop a systematic theory of generalised curvature invariants with 
applications to curvature singularities and extendability in the context of generalised geometry. 
One observation in our paper is that not only
standard supergravity and string theory, but also the generalised supergravity
described in \cite{Arutyunov:2015mqj} fits naturally into our formalism.\footnote{This refers to a deformation of supergravity where an additional vector field is introduced, which is required 
to satisfy a list of conditions, including 
to be a Killing vector field for the metric satisfying the deformed field
equations, see 
\cite{Arutyunov:2015mqj} for details. 
To avoid a conflict of terminology, we will call this deformation 
vector-deformed supergravity instead of generalised supergravity.}  
A natural question, only touched upon in this paper, is the geometrical meaning (within generalised geometry)
of the equations defining such deformed versions of supergravity and whether further 
generalisations are mathematically and physically meaningful.

Finally, we briefly describe the structure of the paper and its main results.  In Section \ref{LC:sec} we define  
\emph{the canonical generalised Levi-Civita connection} $D^{{\mathcal G},\mathrm{div}}$  for a given generalised metric $\mathcal G$ and 
divergence operator $\mathrm{div}$ on an exact Courant algebroid $(E,\pi , \langle \cdot , \cdot  \rangle, [\cdot ,\cdot ])$.  In the special case that $\mathrm{div}$ is the metric divergence operator $\mathrm{div}^{\mathcal G}=\mathrm{div}^{g}\circ \pi$, the canonical generalised Levi-Civita connection $D^0$ is directly obtained from the Levi-Civita connection of the induced semi-Riemannian metric $g$ by extending it to a generalised connection and subtracting the 
torsion.  For the case of general divergence, we add the unique tensor $S$ in the first generalised prolongation 
$\mathfrak{so}(E)_{\mathcal G}^{\langle 1\rangle}$ which is perpendicular to the 
subspace of trace-free tensors and which has the prescribed trace $\mathrm{tr}\, S := (v\mapsto \mathrm{tr}\, S v)= \mathrm{div}-\mathrm{div}^{\mathcal G}\in \Gamma (E^*)$. Thus we obtain the canonical generalised Levi-Civita connection as $D^{{\mathcal G},\mathrm{div}}=D^0+S$. 

In Sections \ref{metric_divergence_section} and \ref{curv_gen:sec} we compute the Riemann tensor of the 
canonical generalised Levi-Civita connection in terms of the ordinary Riemann tensor, the three-form $H$ 
(encoding the Courant algebroid structure) and the generalised dilaton 
\[ e=\eta^{-1}(\mathrm{div}^\mathcal{G}-\mathrm{div}).\] In section \ref{metric_divergence_section}  we treat the case $e=0$ (metric divergence) obtaining 
the pure- and mixed-type components of the generalised Riemann tensor in Theorems \ref{curv_zero_dilat:thm} and \ref{curv_zero_dilat_mixed:thm}, respectively. As part of our calculations, in Proposition \ref{D_can_coeff:prop} and Proposition~\ref{formulaS:prop}, we determine explicit expressions for the 
canonical generalised Levi-Civita connection in terms of the identifications $\pi|_{E_\pm} : E_\pm\cong TM$. In particular, we compute the above tensor $S$ explicitly, hereby relating it to the literature in  Remark~\ref{can_conn:rem}. In Section \ref{curv_gen:sec} we establish the general curvature formulas 
for arbitrary divergence, which constitute the main results of our paper. As part of our complete analysis 
we establish also formulas for the full generalised Ricci tensor, the generalised Ricci tensor and  
the two independent traces of the full generalised Ricci tensor. In Section \ref{Comparison:sec} we provide expressions for the various contractions of the generalised Riemann tensor in terms of components. After specialisation to the cases of standard and of vector-deformed supergravity, we compare 
to expressions obtained previously in the literature. In the case where the divergence operator is encoded in a dilaton function (i.e.\ $X=0$ and $\xi = 2 d\varphi$), we match components of curvature tensors with 
\cite{Coimbra:2011nw} and thus
recover the NS-NS action and field equations of string theory. A similar matching is provided for the case of 
vector-deformed supergravity \cite{Arutyunov:2015mqj}. We show that part of the consistency conditions for vector-deformed supergravity consist in the vanishing of the second generalised scalar curvature. In particular, the second generalised scalar curvature vanishes for metric divergence.

\textbf{Acknowledgements.}   Research of VC is funded by the Deutsche For\-schungs\-gemeinschaft 
(DFG, German Research Foundation) under Germany's Excellence Strategy, EXC 2121 ``Quantum Universe,'' 390833306 and under -- SFB-Gesch\"afts\-zeichen 1624 -- Projektnummer 506632645.
TM would like
to thank the Department of Mathematics of the University of Hamburg for support
and hospitality during several visits when work for this paper was undertaken. The
work of MM is supported by an STFC PGR studentship. The work of TM is supported by the STFC Consolidated
Grant ST/T000988/1. Data access statement: There is no additional data beyond
what is included in this paper.

\newtheorem{Th}{Theorem}[section]  
\newtheorem{Prop}[Th]{Proposition}
\newtheorem{Lem}[Th]{Lemma} 
\newcommand{\bt}{\begin{theorem}\ \ }  
\newcommand{\et}{\end{theorem}}  
\newcommand{\bp}{\begin{proposition}\ \ }  
\newcommand{\ep}{\end{proposition}}  
\newcommand{\bl}{\begin{lemma}\ \ }  
\newcommand{\el}{\end{lemma}}  
\newcommand{\pf}{\noindent{\it Proof:\ \ }}  
\renewcommand{\square}{\kern1pt\vbox  
               {\hrule height 0.6pt\hbox{\vrule width 0.6pt\hskip 3pt  
    \vbox{\vskip 6pt}\hskip 3pt\vrule width 0.6pt}\hrule height0.6pt}  
    \kern1pt}  

\renewcommand{\arraystretch}{1.3}

\newtheorem{note}{Note}
\newenvironment{Note}{\begin{note} \rm }{\end{note} }
\newcommand{\myref}[1]{\ref{#1} on p. \pageref{#1}}


\section{The canonical LC generalised connection}
\label{LC:sec}
Let $(E, \langle \cdot , \cdot \rangle, [\cdot , \cdot ])$ be an exact 
Courant algebroid over $M$ with anchor map $\pi : E \to TM$. 
Let $\mathcal G$ be a (semi-Riemannian) generalised   metric on $E$ with corresponding orthogonal 
decomposition $E=E_+ \oplus E_-$. See \cite[Definition 2.1]{CortesKrusche} for the definition of a 
semi-Riemannian generalised metric. From this definition it follows 
that the restriction $\pi_+ := \pi|_{E_+} : E_+  \to TM$ is an isomorphism.
This gives rise to a semi-Riemannian metric $g$ on $M$ such that $(\pi_+)^*g =  \mathcal G|_{E_+\times E_+}$. Then 
$\pi_- := \pi|_{E_-} : E_-  \to TM$ is also an isomorphism and $(\pi_-)^*g =  \mathcal G|_{E_-\times E_-}$.  
\bt For any divergence operator $\mathrm{div} : \Gamma (E) \to C^\infty (M)$ there exists a canonical 
generalised Levi-Civita connection $D=D^{\mathcal G, \mathrm{div}}$ such that $\mathrm{div}^D = \mathrm{div}$. 
\et 

\begin{proof}
We consider first the special case $\mathrm{div}= \mathrm{div}^\mathcal{G}$, where $\mathrm{div}^\mathcal{G}$ is the so-called metric divergence operator.
It is defined as follows. Let $\phi : E \to 
(\mathbb{T}M, [\cdot , \cdot ]_H)$ be the isomorphism of Courant algebroids determined by the generalised metric, i.e.\ 
\[ \phi|_{E_\pm} = (1\pm g) \circ \pi_\pm.\]
Here $[\cdot , \cdot ]_H$ denotes the standard Dorfman bracket twisted by a closed $3$-form $H$ such that $\phi^*[ \cdot ,\cdot  ]_H = [\cdot , \cdot ]$. 
The metric divergence operator $\mathrm{div}^\mathcal{G} : E \to C^\infty (M)$ is defined by 
\[ \mathrm{div}^\mathcal{G}(v) = \mathrm{tr} ( \nabla_{\pi}  \phi (v)),\quad v\in \Gamma (E), \] 
where $\nabla$ denotes the Levi-Civita connection in $\mathbb{T}M=TM \oplus T^*M$. More explicitly, decomposing $v=v_++v_-$ into 
eigenvectors of $\mathcal G^{\End}$ we have $\phi (v) = X+gX + Y -gY$, where $X=\pi_+(v_+)$, $Y=\pi_-(v_-)$, and hence 
\[ \mathrm{div}^\mathcal{G}(v)  = \mathrm{div}^g(X+Y) =  \mathrm{div}^g(\pi (v)) .\]
So 
\[ \mathrm{div}^\mathcal{G} = \mathrm{div}^g \circ \pi ,\]
where $\mathrm{div}^g  : \Gamma (TM) \to C^\infty (M)$ is the ordinary divergence operator associated with the metric $g$.

We start by defining a metric generalised connection $D^0$ using the Levi-Civita connection $\nabla$ of $g$: 
\begin{eqnarray*} 
D^0_u|_{\Gamma (E_+)} &:=& \pi_+^{-1} \circ \nabla_{\pi (u)}\circ \pi_+, \quad u\in \Gamma (E_+),\\
D^0_u|_{\Gamma (E_-)} &:=& \pi_-^{-1} \circ \nabla_{\pi (u)}\circ \pi_-, \quad u\in \Gamma (E_-),\\
D^0_u|_{\Gamma (E_-)} &:=& [u, \cdot ]_-, \quad u\in \Gamma (E_+),\\
D^0_u|_{\Gamma (E_+)} &:=& [u, \cdot ]_+, \quad u\in \Gamma (E_-),
\end{eqnarray*} 
where the subindex $\pm$ on the bracket denotes projection from $E$ onto $E_\pm$. 
It is  straightforward to compute the torsion $T^0$ of $D^0$ and we obtain that 
$T^0$ is of pure-type, i.e.\ $T^0 \in \Gamma (\wedge^3E_+^* \oplus \wedge^3E_-^*)$ and is given by 
\[ T^0(u,v,w) = - \langle [u,v],w\rangle \mp \pi (v) g (\pi (u), \pi (w))\pm \sum_{\mathrm{cycl}} g(\nabla_{\pi (u)}\pi (v),\pi (w))  ,\]
whenever $u,v,w\in \Gamma (E_+)$ or $u,v,w\in \Gamma (E_-)$. 
Now we define a canonical Levi-Civita generalised connection by
\[ D := D^0 -\frac13 T^0.\]
Since $T^0$ is totally skew-symmetric it is trace-free and $\mathrm{div}^D = \mathrm{div}^{D^0}$.
The latter obviously coincides with the divergence operator of the metric generalised connection 
\begin{equation} \label{pm_gen_conn:eq}\pi_+^{-1} \circ \nabla_{\pi }\circ \pi_+ + \pi_-^{-1} \circ \nabla_{\pi }\circ \pi_-.\end{equation} 
We claim that it coincides with the metric divergence operator $\mathrm{div}^\mathcal{G}$. 
We compute the divergence operator of \eqref{pm_gen_conn:eq} evaluated on $v=v_++v_-\in \Gamma (E)$ 
in terms of the vector fields $X=\pi_+(v_+), Y=\pi_-(v_-)$:
\begin{eqnarray*}\mathrm{div}^D(v) &=& \mathrm{tr}(\pi_+^{-1} \circ \nabla_{\pi }X + \pi_-^{-1} \circ \nabla_{\pi }Y )\\
&=& \mathrm{tr}_{E_+} (\pi_+^{-1} \circ \nabla_{\pi_+ }X) + 
\mathrm{tr}_{E_-} (\pi_+^{-1} \circ \nabla_{\pi_- }Y)\\ 
&=& \mathrm{div}^g (X) + \mathrm{div}^g (Y) = \mathrm{div}^g(\pi (v))=\mathrm{div}^\mathcal{G} (v).\end{eqnarray*}
We have proven that $D$ is a canonical generalised Levi-Civita  connection with $\mathrm{div}^D = \mathrm{div}^\mathcal{G}$. 

Next we consider an arbitrary divergence operator 
\[ \mathrm{div}= \mathrm{div}^\mathcal{G} + \alpha,\quad \alpha \in \Gamma (E^*).\]
Any generalised Levi-Civita  connection with the divergence $\mathrm{div}$ is of the 
form $D+S$, where $S\in \mathfrak{so}(E)_{\mathcal G}^{\langle 1\rangle}$ has $\mathrm{tr} (Sv) = \alpha (v)$ for all 
$v\in E$. Here 
$\mathfrak{so}(E)_{\mathcal G}$ denotes the stabilizer of $\mathcal G$ in $\mathfrak{so}(E)$ and 
$\mathfrak{so}(E)_{\mathcal G}^{\langle 1\rangle}:=(\mathfrak{so}(E)_{\mathcal G})^{\langle 1\rangle}$ denotes its first generalised prolongation \cite{cortesdavid}. 
This first generalised prolongation is described in detail in the proof of Lemma \ref{nondeg:lem} below. 
We denote the map $v\mapsto \mathrm{tr}(Sv)$ simply by $\mathrm{tr}(S)$. 
In the next lemma we show that the kernel of the map 
\begin{equation} \label{trace:eq}\mathrm{tr} : \mathfrak{so}(E)_{\mathcal G}^{\langle 1\rangle} \to E^*,\quad S \mapsto \mathrm{tr}(S)\end{equation}
is nondegenerate with respect to the scalar product $\langle \cdot, \cdot \rangle$. 
Therefore its orthogonal bundle 
\[ \mathcal K := (\ker  \mathrm{tr})^{\perp}\]
is a canonical complementary subbundle. The bundle $\mathcal K$ 
has a unique section $S^\alpha$ such that $\mathrm{tr}(S^\alpha)=\alpha$. 
This proves that $D+S^\alpha$ is a canonical generalised Levi-Civita  connection with $\mathrm{div}^{D+S_\alpha}=\mathrm{div}$. 
\end{proof}
\bl \label{nondeg:lem} The kernel of the trace-map 
\[ \mathrm{tr} : \mathfrak{so}(E)_{\mathcal G}^{\langle 1\rangle} \to E^*,\quad S \mapsto \mathrm{tr}(S)\]
is nondegenerate with respect to the metric induced by $\mathcal G$ and with respect to the metric induced by $\langle \cdot, \cdot \rangle$ as well. 
\el 

\begin{proof} 
Since  $\mathfrak{so}(E)_{\mathcal G}^{\langle 1\rangle} = \mathfrak{so}(E_+)^{\langle 1\rangle}\oplus \mathfrak{so}(E_-)^{\langle 1\rangle}$, it 
is sufficient to consider the two restrictions 
\[ \mathrm{tr} : \mathfrak{so}(E_\pm)^{\langle 1\rangle} \to E_\pm^*.\]
The proof is analogous for both of them and reduces to the following algebraic statement. 
Consider a pseudo-Euclidean vector space $V$. Then we claim that the kernel of 
\begin{equation} \label{tr:eq}\mathrm{tr} : \mathfrak{so}(V)^{\langle 1\rangle} \to V^*\end{equation}
is nondegenerate with respect to the induced scalar product. 
To prove this we decompose $V=L \oplus P$ as an orthogonal 
sum of a negative definite subspace $L$  and a positive definite
subspace $P$. We decompose
\begin{eqnarray} \label{one:eq}
V^* \otimes \wedge^2 V^* &=&
(L^* \otimes \wedge^2 L^*) \oplus
(L^* \otimes \wedge^2 P^*) \oplus
(P^* \otimes L^* \wedge P^*)   \oplus\\
&& \label{two:eq}
(L^* \otimes L^* \wedge P^*) \oplus
(P^* \otimes \wedge^2 L^*) \oplus
(P^* \otimes \wedge^2 P^*),
\end{eqnarray}
where the induced scalar product is negative definite on 
terms in the first line and positive definite on terms in the
second line. Next we show that 
\[ \mathfrak{so}(V)^{\langle 1\rangle} = \{ S \in V^* \otimes \wedge^2 V^*\mid \partial S =0 \} \subset V^* \otimes \wedge^2 V^*\] 
is a nondegenerate 
subspace, where $\partial : V^* \otimes \wedge^2 V^* \to \wedge^3 V^*$ is given by 
\[ (\partial S)(u,v,w)=\sum_{\mathrm{cycl}} S(u,v,w).\]  
It clear that the equation $\partial S =0$ decouples into four independent equations corresponding to the 
tensor power of $L$ in the decomposition \eqref{one:eq}-\eqref{two:eq}. Since \eqref{one:eq} collects the 
terms of odd degree in $L$ and \eqref{two:eq} those of even degree, the equation $\partial S =0$ defines 
a subspace which is a direct sum of a subspace $\mathfrak{so}(V)^{\langle 1\rangle}_{\mathrm{odd}}$ of \eqref{one:eq}  and a subspace  $\mathfrak{so}(V)^{\langle 1\rangle}_{\mathrm{even}}$ of \eqref{two:eq}. 
The latter subspaces are definite and therefore their sum is nondegenerate. Finally, the kernel of \eqref{tr:eq}
splits as 
\[ \ker \left(\mathrm{tr} : \mathfrak{so}(V)^{\langle 1\rangle}_{\mathrm{odd}} \to L^*\right) \oplus  \ker \left(\mathrm{tr} : \mathfrak{so}(V)^{\langle 1\rangle}_{\mathrm{even}} \to P^*\right) \]
and is therefore as well a sum of two definite subspaces. 
\end{proof} 
\section{Curvature computations in the case of metric divergence}\label{metric_divergence_section}
Recall (see \cite{vicenteoskar} and references therein) that the generalised Riemann tensor of a generalised Levi-Civita connection $D$ on $(E,\mathcal G)$ is defined as follows:
\begin{equation}\label{genriemdefeq}
    \begin{split}
        &\genriem^D(a,b,v,w) \coloneqq \langle \genriem^D(v,w)b,a\rangle\\ 
        &:= \frac{1}{2} \Big \{ \scalbrack{(D^2_{v,w}-D^2_{w,v})b, a} + \scalbrack{(D^2_{b,a}-D^2_{a,b})v,w} 
         - \trwith{E} (\scalbrack{D v, w} \scalbrack{Db,a})\Big\},\\
    \end{split}
    \end{equation}
    where $a,b,v,w\in \Gamma (E)$. It is a section of $\Sym^2\,\Lambda^2\,E^*$ and satisfies the Bianchi identity.
    We will use that it is uniquely determined by its pure-type components in $\Sym^2\,\Lambda^2\,E^*_+\oplus \Sym^2\,\Lambda^2\,E^*_-$ and its mixed-type components in $(E_+^*\wedge E_-^*)\vee \Lambda^2\,E^*_+ \oplus 
    (E_-^*\wedge E_+^*)\vee \Lambda^2\,E^*_-$. 
The proof of the following proposition is straightforward.
\begin{proposition}
Let $D$ be the canonical generalised Levi-Civita connection with metric divergence. 
Then we have the following connection coefficients for any sections $u ,v , w\in \Gamma (E)$ 
with constant scalar products. 
\begin{enumerate}
\item
Pure type coefficients, i.e.\ $u ,v , w\in \Gamma (E_+)$ or $u ,v , w\in \Gamma (E_-)$:
\[ \langle D_uv,w\rangle = \left(\pm \Gamma + \frac13 \mathcal B \mp\frac13 \partial \Gamma\right)(u,v,w),\]
where we have abbreviated $\mathcal B(u,v,w) := \langle [u,v],w\rangle$, 
$\Gamma (u,v,w ) := g(\nabla_{\pi u}\pi v,\pi w)$ and $(\partial \Gamma) (u,v,w ) = \sum_{\mathrm{cycl}}\Gamma (u,v,w)$.
\item Mixed-type coefficients, i.e.\ $u\in \Gamma (E_\pm)$ and $\Bar{v},\Bar{w} \in \Gamma (E_\mp )$: 
\[ \langle D_u\Bar{v},\Bar{w}\rangle =  \mathcal B (u,\Bar{v},\Bar{w}).\]
\end{enumerate}
Without evaluating on $w$ these formulas can be written as follows. 
Pure type coefficients: 
\[ D_uv = \pi_{\pm}^{-1}\nabla_{\pi u}\pi v + \frac13 [u,v]_\pm -\frac13 (\pi_{\pm}^{-1}\nabla_{\pi u}\pi v  - \pi_{\pm}^{-1}\nabla_{\pi v}\pi u 
- \mathcal{G}^{-1} g(\nabla_{\pi_\pm } \pi u , \pi v)).\]
The last term can be rewritten as 
 \[ \mathcal{G}^{-1} g(\nabla_{\pi_\pm} \pi u , \pi v) = \pi_{\pm}^{-1}g^{-1}\left(g(\nabla \pi u , \pi v)\right)\in \Gamma (E).\] 
 Mixed-type coefficients:
 \[ D_u\Bar{v} =  [u,\Bar{v}]_\mp .\]
\ep
We rewrite the above formulas such that the Dorfman coefficients $\mathcal B(u,v,w)$ are eliminated by
expressing them in terms of the three-form coefficients $(\pi^*H)(u,v,w) = H(\pi u , \pi v, \pi w)$ and the 
connection coefficients $\Gamma (u,v,w ) = g(\nabla_{\pi u}\pi v,\pi w)$ of $\nabla$. 
\bp \label{D_can_coeff:prop}Let $D$ be the canonical generalised Levi-Civita connection with metric divergence. 
Then we have the following connection coefficients for any sections $u ,v , w\in \Gamma (E)$ 
with constant scalar products. 
\begin{enumerate}
\item Pure-type coefficients:
\[ \langle D_uv,w\rangle = \left( \frac16 (\pi^*H) \pm  \Gamma\right)(u,v,w),\quad u,v,w\in \Gamma (E_\pm).\]
\item Mixed-type coefficients: 
\[  \langle D_u\Bar{v},\Bar{w}\rangle = \left(\frac12 (\pi^*H) -  \Gamma\right)(u,\Bar{v},\Bar{w}),\quad u\in \Gamma (E_+),\quad \Bar{v},\Bar{w}\in \Gamma (E_- ),\]
\[  \langle D_u\Bar{v},\Bar{w}\rangle = \left(\frac12 (\pi^*H) +\Gamma \right)(u,\Bar{v},\Bar{w}),\quad u\in \Gamma (E_-),\quad \Bar{v},\Bar{w}\in \Gamma (E_+ ).\]
\end{enumerate}
\end{proposition} 
\begin{proof} For $u,v,w\in \Gamma (E_\pm)$ we have
\[ \langle D_uv,w\rangle = \left( \pm \Gamma + \frac13 \mathcal B \mp \frac13 \partial \Gamma\right)(u,v,w).\]
To compute $\mathcal B(u,v,w)$ we begin by computing $[u,v]$ in terms of the vector fields $X=\pi u$, $Y=\pi v$ using that $g(X,Y)$ is constant and 
$(\mathcal{L}_Xg)Y = 2 g ( (\nabla X)^{\mathrm{sym}}Y)$: 
\begin{eqnarray*} [u,v] &=& [X\pm gX,Y\pm gY] =\mathcal{L}_X(Y\pm gY)\mp \iota_Yd(gX)+H(X,Y,\cdot )\\ 
&=&\mathcal{L}_X(Y\pm gY)\mp \mathcal{L}_Y(gX)+H(X,Y,\cdot )\\
&=& \mathcal{L}_XY \pm (\mathcal{L}_Xg)Y\mp (\mathcal{L}_Yg)X \pm g(\mathcal{L}_XY-\mathcal{L}_YX) + H(X,Y,\cdot )\\
&=&  \mathcal{L}_XY \pm  2 g ( (\nabla X)^{\mathrm{sym}}Y) \mp 2 g ( (\nabla Y)^{\mathrm{sym}}X)\pm  2g(\mathcal{L}_XY) + H(X,Y,\cdot ).
\end{eqnarray*}
Taking the scalar product with $w=Z\pm gZ$ and using that $g(\mathcal{L}_XY,Z)=\Gamma (u,v,w) -\Gamma (v,u,w)$ we obtain:
\begin{eqnarray*} \mathcal B(u,v,w) &=&  \pm \langle \mathcal{L}_XY , gZ\rangle \pm  2 \langle g ( (\nabla X)^{\mathrm{sym}}Y),Z\rangle \mp 2
\langle g ( (\nabla Y)^{\mathrm{sym}}X),Z\rangle\\
&& \pm \; 2 \langle g(\mathcal{L}_XY), Z\rangle + \langle H(X,Y,\cdot ), Z\rangle\\
&=& \pm \frac12 g(\mathcal{L}_XY,Z) \pm  g ( (\nabla X)^{\mathrm{sym}}Y,Z) \mp g( (\nabla Y)^{\mathrm{sym}}X,Z)\\ &&\pm \; g(\mathcal{L}_XY,Z)+\frac12 H(X,Y,Z)\\
&=& \pm \frac32  g(\mathcal{L}_XY,Z)  \pm \frac12(\Gamma (v, u, w)+ \Gamma (w, u, v) - \Gamma (u, v, w)-\Gamma (w, v,u))\\ &&+\;\frac12(\pi^*H)(u,v,w)\\
&=& \pm \left( \Gamma (u,v,w) -\Gamma (v,u,w)+\Gamma (w,u,v)\right) + \frac12(\pi^*H)(u,v,w)\\
&=& \left( \pm \partial \Gamma + \frac12(\pi^*H)\right)(u,v,w).
\end{eqnarray*}
This implies the claimed formula for the pure type coefficients. The formulas for the mixed-type coefficients can be obtained similarly. 
\end{proof} 
\begin{remark}\label{can_conn:rem}
It follows from the formulas of Proposition \ref{D_can_coeff:prop} that the canonical generalised Levi-Civita connection $D$ with metric divergence  coincides with the generalised Levi-Civita connection of \cite{genricciflow} constructed out of the (torsion-full) Gualtieri-Bismut generalised connection. Next we show that the
canonical generalised Levi-Civita connection
$D^{\mathcal{G},
\mathrm{div}}$ with divergence $\mathrm{div}$ 
has the form  
$D+S$, where $S$ 
is precisely the  tensor \eqref{Chi-tensor:eq} considered in \cite{genricciflow}. This shows that the 
generalised Levi-Civita connection $D+S$ is canonically associated with the pair $(\mathcal{G},
\mathrm{div})$. 
\end{remark}
\begin{proposition} \label{formulaS:prop} Let $(\mathcal G, \mathrm{div})$ be a generalised metric and  divergence operator on an exact Courant algebroid $E\to M$, $D$ the canonical generalised Levi-Civita connection with metric divergence and $e\in \Gamma (E)$ defined by $\langle e, \cdot \rangle =\mathrm{div}^\mathcal{G}-\mathrm{div}$. Then the tensor 
$S= D^{\mathcal{G},
\mathrm{div}}-D$
is given by 
\begin{equation} \label{Chi-tensor:eq} S=\frac{\chi_+^{e_+} + \chi_-^{e_-}}{\dim M - 1},\end{equation} 
where $\chi_\pm^{e_\pm} \in \Gamma(\mathfrak{so}(E_\pm)^{\langle 1\rangle })$ is defined by      
            \[ \chi_\pm^{e_\pm}(a,b) \coloneqq \scalbrack{a,b} e_\pm - a \scalbrack{e_\pm, b}, \qquad\quad a,b \in \Gamma(E_\pm),
         \] 
         in terms of the projections $e_\pm \in \Gamma (E_\pm)$ of $e$.
\end{proposition} 
\begin{proof}
We need to check that $S$ given by \eqref{Chi-tensor:eq} satisfies 
$\langle S,T\rangle =0$ for all 
\[ T\in \ker\, ( \mathrm{tr} : \mathfrak{so}(E)^{\langle 1\rangle }_\mathcal{G}\to E^*).\] 
It suffices to check this for $T\in \mathfrak{so}(E_\pm )^{\langle 1\rangle }_0$, where the 
index zero stands for trace-free tensors. The proofs for $E_+$ and $E_-$ being similar, we restrict to 
$E_+$ and compute the scalar product  by means of an orthonormal basis $(u_i)$  of $E_+|_p$ at some point $p\in M$:  
\begin{eqnarray*} \langle S,T\rangle &=& \sum_{ijk} \chi_+ (u_i,u_j,u_k) T_{ijk}\varepsilon_i\varepsilon_j\varepsilon_k\\
&=& \sum_{ijk}(\varepsilon_i\varepsilon_k\delta_{ij}e_+^k - \varepsilon_i\varepsilon_j\delta_{ik}e_+^j)T_{ijk}\varepsilon_i\varepsilon_j\varepsilon_k\\
&=& \sum_{ijk} T_{ijk}(\varepsilon_j\delta_{ij}e_+^k-\varepsilon_k\delta_{ik}e_+^j)\\
&=& 0,\end{eqnarray*}
where $\chi_+=\chi_+^{e_+}$, $T_{ijk}=T(u_i,u_j,u_k)$, $e_+^i=\varepsilon_i\langle e_+,u_i\rangle$ and $\varepsilon_i = \langle u_i,u_i\rangle$.  In the last step we used that $T$ is trace-free. 
\end{proof}

Recall that the generalised Riemann tensor is a sum of its pure-type part in $\Sym^2\,\Lambda^2\,E^*_+\oplus \Sym^2\,\Lambda^2\,E^*_-$ and its mixed-type part in $(E_+^*\wedge E_-^*)\vee \Lambda^2\,E^*_+ \oplus 
    (E_-^*\wedge E_+^*)\vee \Lambda^2\,E^*_-$. These two components are determined separately in the following 
    two theorems. 
\bt \label{curv_zero_dilat:thm} The pure-type part of the Riemann tensor of the canonical generalised Levi-Civita connection $D$ with metric divergence is given by the following 
formulas.
For all $a,b,v,w \in \Gamma (E_+)$ we have
\begin{eqnarray*} \genriem^{D}(a,b,v,w) &=& g( \mathrm{Rm}(v,w)b,a) + \frac{1}{36}H^{(2)}(a,v,w,b)\\ 
&&+\;\frac{1}{36}H^{(2)}(v,b,w,a) + 
\frac{1}{18} H^{(2)}(v,w,b,a),\end{eqnarray*}
where $H^{(2)}$ denotes the contraction of $H^{\otimes 2}\in \Gamma (\mathrm{Sym}^2\bigwedge^3T^*M)$ in the 
first and fourth argument using the metric $g$ and we use the notation $F(a,b,v,w) := F(\pi a, \pi b, \pi v, \pi w)$
for any tensor $F$ on $M$.  Similarly, for all $a,b,v,w \in \Gamma (E_-)$ we have 
\begin{eqnarray*} - \genriem^{D}(a,b,v,w) &=& g( \mathrm{Rm}(v,w)b,a) + \frac{1}{36}H^{(2)}(a,v,w,b)\\ 
&&+\;\frac{1}{36}H^{(2)}(v,b,w,a) + 
\frac{1}{18} H^{(2)}(v,w,b,a).\end{eqnarray*}
More compactly, we can write both cases in one formula as 
\begin{eqnarray*} \pm \genriem^{D}(a,b,v,w) &=& g( \mathrm{Rm}(v,w)b,a) + \frac{1}{36}H^{(2)}(a,v,w,b)\\ 
&&+\;\frac{1}{36}H^{(2)}(v,b,w,a) + 
\frac{1}{18} H^{(2)}(v,w,b,a),\end{eqnarray*}
where now the pure-type insertions are $a,b,v,w\in \Gamma (E_\pm)$. 
\et
\begin{proof}
We consider only the case $a,b,v,w\in \Gamma (E_+)$, since the other case is similar. We will compute each term in the following formula: 
\begin{eqnarray*} 2 \genriem^{D}(a,b,v,w) &=& \mathfrak{S} \langle D_vD_wb,a\rangle - \mathfrak{S} \langle D_{D_vw}b,a\rangle -\langle (Dv)^*w,(Db)^*a\rangle\\
&=&\mathfrak{S} \pi (v)\langle D_wb,a\rangle - \mathfrak{S} \langle D_wb,D_va\rangle - \mathfrak{S} \langle D_{D_vw}b,a\rangle\\
&& -\;\langle (Dv)^*w,(Db)^*a\rangle, 
 \end{eqnarray*}
where we use the notation 
\[ \mathfrak{S} L(v,w,b,a) := L(v,w,b,a) - L(w,v,b,a) + L(b,a,v,w) -L(a,b,v,w)\] 
for any differential operator $L : \Gamma (E^{\otimes 4}) \to C^{\infty}(M)$. Note in the second line $\mathfrak{S} \langle D_wb,D_va\rangle$ stands for 
$\mathfrak{S} L(v,w,b,a)$ with $L(v,w,b,a) = \langle D_wb,D_va\rangle$. 

For the calculations we can assume that the sections $a,b,v,w$ have constant scalar products and vanishing $D$-derivatives at
the point $p$ at which we compute the Riemann tensor. 
We begin with the case $a,b,v,w\in \Gamma (E_+)$ and compute term by term. From Proposition~\ref{D_can_coeff:prop} we have 
\[ \pi (v)\langle D_wb,a\rangle = \frac16 v H(w,b,a) + v\Gamma (w,b,a),\]
where $\Gamma (w,b,a):= g(\nabla_wb,a)$. At $p$ we obtain 
\begin{eqnarray*} \mathfrak{S} \pi (v)\langle D_wb,a\rangle|_p &=& \frac16 (dH)(v,w,b,a)|_p +2g( \mathrm{Rm}(v,w)b,a)|_p\\
&=& 2g( \mathrm{Rm}(v,w)b,a)|_p.\end{eqnarray*}
We compute the next term using an arbitrary local frame $(e_I) =(e_i, e_{\hat{j}})$ of $E$ adapted to the decomposition $E=E_+ \oplus E_-$ and the 
notation $(e^I)=(e^i, e^{\hat{j}})$ for the frame such that $\langle e^I, e_J\rangle = \delta^I_J$: 
\[ \langle D_wb,D_va\rangle = \sum_{I=1}^{2n} \langle D_wb,e_I\rangle \langle D_va, e^I\rangle = 
\sum_{i=1}^n \langle D_wb,e_i\rangle \langle D_va, e^i\rangle .\]
At $p$ this evaluates to 
\[ \langle D_wb,D_va\rangle|_p = \frac{1}{36}H^{(2)}(w,b,v,a)|_p .\]
Hence 
\begin{eqnarray*} -\mathfrak{S} \langle D_wb,D_va\rangle|_p &=&  -\frac{1}{36}H^{(2)}(w,b,v,a)|_p + \frac{1}{36}H^{(2)}(v,b,w,a)|_p\\
&& -\; \frac{1}{36}H^{(2)}(a,v,b,w)|_p 
+\frac{1}{36}H^{(2)}(b,v,a,w)|_p\\
&=& -\frac{1}{18}H^{(2)}(w,b,v,a)|_p + \frac{1}{18}H^{(2)}(v,b,w,a)|_p.\end{eqnarray*} 
Next we compute 
\[ \langle D_{D_vw}b,a\rangle = \sum_{I=1}^{2n}\langle {D_vw},e_I\rangle \langle D_{e^I}b,a\rangle  = \sum_{i=1}^n \langle {D_vw},e_i\rangle \langle D_{e^i}b,a\rangle,
 \] 
 which evaluates at $p$ to 
 \[ \langle D_{D_vw}b,a\rangle|_p = \frac{1}{36} H^{(2)}(v,w,b,a)|_p.\]
 Thus 
 \[ - \mathfrak{S} \langle D_{D_vw}b,a\rangle|_p = -\frac{1}{9}H^{(2)}(v,w,b,a)|_p.\] 
 Finally we compute 
 \begin{eqnarray*} \langle (Dv)^*w,(Db)^*a\rangle &=& \sum_{I=1}^{2n} \langle D_{e_I}v,w\rangle \langle D_{e^I}b,a\rangle\\
  &=& \sum_{i=1}^{n} \langle D_{e_i}v,w\rangle \langle D_{e^i}b,a\rangle + \sum_{j=1}^{n} \langle D_{e_{\hat j}}v,w\rangle \langle D_{e^{\hat j}}b,a\rangle .
  \end{eqnarray*} 
Using  Proposition~\ref{D_can_coeff:prop} and evaluating at $p$ we get
\begin{eqnarray*} \langle (Dv)^*w,(Db)^*a\rangle|_p &=& \frac{1}{36}H^{(2)}(v,w,b,a)|_p - \frac14 H^{(2)}(v,w,b,a)|_p\\
 &=& -\frac29 H^{(2)}(v,w,b,a)|_p,\end{eqnarray*}
where the minus sign is due to the fact that $\pi_- : E_-  \to (TM,g)$ is an isometry with respect to $\mathcal G$ and therefore an 
anti-isometry with respect to $\langle \cdot , \cdot \rangle$. 
This finishes the proof in the case $a,b,v,w\in \Gamma (E_+)$. 
\end{proof}

The following theorem can be extracted from \cite[§ 3.4]{genricciflow}. For completeness, we give an independent proof using the conventions and results of our paper. 
\bt \label{curv_zero_dilat_mixed:thm}
    The mixed-type components of the generalised Riemann curvature are given by
    \begin{equation*}
        \begin{split}
            &\pm2 \genriem^D(a,\Bar{b},v,w) \\
            &= \riem(a,\Bar{b},v,w) \mp \frac{1}{2}[\nabla_{a} H](\Bar{b},v,w) \pm \frac{1}{6}[\nabla_{\Bar{b}} H](a,v,w) \\
            &\quad - \frac{1}{12} H^{(2)}(\Bar{b},w,a,v) - \frac{1}{12} H^{(2)}(w,a,\Bar{b},v) - \frac{1}{6} H^{(2)}(a,\Bar{b},v,w), \\            
        \end{split}
    \end{equation*}
    where $a,v,w \in \Gamma(E_\pm)$ and $\Bar{b}\in \Gamma(E_\mp)$ arbitrary.
    \et
\begin{proof}
    We have to compute
    \begin{equation*}
        \pm2\genriem^D(a,\Bar{b},v,w) = \pm\scalbrack{D^2_{\Bar{b},a}v - D^2_{a,\Bar{b}}v,w}. \\    
    \end{equation*}
    We start with the first term.
    \begin{equation*}
        \begin{split}
            &\pm\scalbrack{D^2_{\Bar{b},a} v,w} \\
            &= \pm\scalbrack{D_{\Bar{b}}D_a v - D_{D_{\Bar{b}}a}v,w} \\
            &= g(\nabla_{\Bar{b}} D_a v,w) \pm \frac{1}{2}H(\Bar{b},D_a v, w) - g(\nabla_{D_{\Bar{b}}a}v,w) \mp \frac{1}{6}H(D_{\Bar{b}}a,v,w) \\
            &= g(\nabla^2_{\Bar{b},a} v,w)\pm \frac{1}{6}[\nabla_{\Bar{b}}(H(a,v))](w)  \pm \frac{1}{2} H(\Bar{b}, \nabla_a v, w) - \frac{1}{12}H^{(2)}(\Bar{b},w,a,v) \\
            &\quad  \mp \frac{1}{2} g(\nabla_{H(\Bar{b},a)}v,w) \mp \frac{1}{6}H(\nabla_{\Bar{b}}a,v,w)- \frac{1}{12}H^{(2)}(\Bar{b},a,v,w) \\
            &= g(\nabla^2_{\Bar{b},a} v,w)\mp \frac{1}{6}[\nabla_{\Bar{b}}(H(v))](a,w)  \pm \frac{1}{2} H(\Bar{b}, \nabla_a v, w) - \frac{1}{12}H^{(2)}(\Bar{b},w,a,v) \\
            &\quad  \mp \frac{1}{2} g(\nabla_{H(\Bar{b},a)}v,w) - \frac{1}{12}H^{(2)}(\Bar{b},a,v,w)
        \end{split}
    \end{equation*}
    Similarly
    \begin{equation*}
        \begin{split}
            &\pm\scalbrack{D^2_{a,\Bar{b}} v,w}\\
            &= \pm\scalbrack{D_aD_{\Bar{b}} v - D_{D_a{\Bar{b}}}v,w} \\
            &= g(\nabla_a D_{\Bar{b}} v,w) \pm \frac{1}{6}H(a,D_{\Bar{b}} v, w) - g(\nabla_{D_a {\Bar{b}}}v,w) \mp \frac{1}{2}H(D_a{\Bar{b}},v,w) \\
            &= g(\nabla^2_{a,\Bar{b}} v,w)\pm \frac{1}{2}[\nabla_a(H({\Bar{b}},v))](w)  \pm \frac{1}{6} H(a, \nabla_{\Bar{b}} v, w) - \frac{1}{12}H^{(2)}(a,w,{\Bar{b}},v) \\
            &\quad  \pm \frac{1}{2} g(\nabla_{H(a,\Bar{b})}v,w) \mp \frac{1}{2}H(\nabla_a{\Bar{b}},v,w)+ \frac{1}{4}H^{(2)}(a,\Bar{b},v,w) \\
            &= g(\nabla^2_{a,\Bar{b}} v,w)\mp \frac{1}{2}[\nabla_a(H(v))](\Bar{b},w)  \pm \frac{1}{6} H(a, \nabla_{\Bar{b}} v, w) - \frac{1}{12}H^{(2)}(a,w,{\Bar{b}},v) \\
            &\quad  \pm \frac{1}{2} g(\nabla_{H(a,\Bar{b})}v,w) + \frac{1}{4}H^{(2)}(a,\Bar{b},v,w) \\
        \end{split}
    \end{equation*}
    The result follows.
\end{proof}
\section{Curvatures for arbitrary divergence operator}
\label{curv_gen:sec} 
Let $E\to M$ be an exact Courant algebroid with semi-Riemannian generalised metric $\genmet$ and divergence operator $\divergence = \divergence^\genmet - \scalbrack{e, \cdot}$. In this section, we compute the components of the generalised Riemann tensor for the canonical generalised Levi-Civita connection $D=D^{\mathcal{G},\mathrm{div}}$. 

We write $D$ as $D^0+S$, where 
$D^0$ is the canonical generalised Levi-Civita connection with metric divergence and 
$S\in \Gamma (\mathfrak{so}(E)^{\langle 1\rangle }_\mathcal{G})$ is the unique section   
orthogonal to $\ker\, (\mathrm{tr} :  \mathfrak{so}(E)^{\langle 1\rangle }_\mathcal{G}\to E^*)$ such that $\mathrm{div}^\mathcal{G}+\mathrm{tr}\, S = \mathrm{div}$. According to Proposition \ref{formulaS:prop}, the tensor $S$ is given by the formula \eqref{Chi-tensor:eq}. 

From the decomposition $D = D^0 + S$, it follows that $\genriem^D$ is given by $\genriem^{D^0}$ plus some extra terms that involve $S$.  We already computed $\genriem^{D^0}$ in Section \ref{metric_divergence_section}. Thus, to obtain formulas for $\genriem^D$, it remains to calculate the extra terms. There are exactly two types of extra terms: terms involving a covariant derivative of $S$ and terms which are quadratic and algebraic in $S$. 
\begin{theorem}
    \label{puretyperiemprop}
    The pure-type components of the generalised Riemann curvature are given by
    \begin{equation*}
        \begin{split}
            &\pm\genriem^{D}(a,b,v,w) \\
            &= \riem(a,b,v,w)  - \frac{1}{36} H^{(2)}(a,v,b,w) - \frac{1}{36}H^{(2)}(b,v,w,a) - \frac{1}{18} H^{(2)}(v,w,a,b) \\
            &\quad \pm \frac{1}{2(d-1)} \left\{ [D^0_v \chi_\pm^{e_\pm}](w,b,a) - [D^0_w \chi_\pm^{e_\pm}](v,b,a)\right. \\
            &\qquad\qquad\qquad \left. + [D^0_b \chi_\pm^{e_\pm}](a,v,w) - [D^0_a \chi_\pm^{e_\pm}](b,v,w) \right\} \\
            &\quad + \frac{1}{2(d-1)^2} \left\{2 \genmet(e_\pm,e_\pm) \big[\genmet(w,a)\genmet(v,b)-\genmet(v,a)\genmet(w,b)\big] \right. \\
            &\qquad\qquad\qquad + \genmet(a,e_\pm) \big[\genmet(w,b)\genmet(v,e_\pm)-\genmet(v,b)\genmet(w,e_\pm)\big] \\
            &\qquad\qquad\qquad + \left.\genmet(b,e_\pm) \big[\genmet(v,a)\genmet(w,e_\pm)-\genmet(w,a)\genmet(v,e_\pm)\big]\right\} .
        \end{split}
    \end{equation*}
    Herein, $a,b,v,w \in \Gamma(E_\pm)$ arbitrary, we denoted $d \coloneqq \dim M$, and from now on we use the notation $\riem(a,b,v,w) \coloneqq \riem (\pi a,\pi b, \pi v, \pi w)$. 
\end{theorem}
\begin{proof}
    We calculate
    \begin{equation*}
        \begin{split}
            &D_v D_w b \\
            &= \left(D^0 + \frac{1}{d-1}\chi_\pm^{e_\pm}\right)_v \left(D^0 + \frac{1}{d-1}\chi_\pm^{e_\pm}\right)_w b \\
            &= D^0_v D^0_w b + \frac{1}{d-1}\chi_\pm^{e_\pm}(v, D^0_w b) + \frac{1}{d-1} D^0_v[\chi_\pm^{e_\pm}(w,b)] + \frac{1}{(d-1)^2} \chi_\pm^{e_\pm}(v,\chi_\pm^{e_\pm}(w,b)). 
        \end{split}
    \end{equation*}
    Similarly, 
    \begin{equation*}
        \begin{split}
            &D_{D_v w} b -  D^0_{D^0_v w} b \\
            &= \frac{1}{d-1}\chi_\pm^{e_\pm}(D^0_v w, b) + \frac{1}{d-1} D^0_{\chi_\pm^{e_\pm}(v,w)}b + \frac{1}{(d-1)^2} \chi_\pm^{e_\pm}(\chi_\pm^{e_\pm}(v,w),b) 
        \end{split}
    \end{equation*}
    and
    \begin{equation*}
        \begin{split}
            & \trwith{E}(\scalbrack{Dv,w}\scalbrack{Db,a}) - \trwith{E}(\scalbrack{D^0v,w}\scalbrack{D^0b,a}) \\
            & = \frac{1}{d-1}\trwith{E}\chi_\pm^{e_\pm}(\cdot,v,w)\scalbrack{D^0b,a} + \frac{1}{d-1} \trwith{E}\scalbrack{D^0v,w}\chi_\pm^{e_\pm}(\cdot,b,a) \\
            &\quad +\frac{1}{(d-1)^2} \trwith{E}\chi_\pm^{e_\pm}(\cdot,v,w)\chi_\pm^{e_\pm}(\cdot,b,a). \\
        \end{split}
    \end{equation*}
    Therefore, employing that $\sum_{\sigma(a,b,c)}\chi_\pm^{e_\pm}(a,b,c) = 0$,
    \begin{equation*}
        \begin{split}
            &2\genriem^D(a,b,v,w) -2\genriem^{D^0}(a,b,v,w) \\
            &= \frac{1}{d-1} \left\{ [D^0_v \chi_\pm^{e_\pm}](w,b,a) - [D^0_w \chi_\pm^{e_\pm}](v,b,a) + [D^0_b \chi_\pm^{e_\pm}](a,v,w) - [D^0_a \chi_\pm^{e_\pm}](b,v,w) \right\} \\
            &\quad + \frac{1}{d-1} \trwith{E}\left\{\scalbrack{D^0b,a} \left[\chi_\pm^{e_\pm}(w,v, \cdot)  - \chi_\pm^{e_\pm}(v,w, \cdot) - \chi_\pm^{e_\pm}(\cdot, v,w) \right]\right.\\
            &\qquad\qquad\qquad \left. +\scalbrack{D^0v,w} \left[\chi_\pm^{e_\pm}(a,b, \cdot)  - \chi_\pm^{e_\pm}(b,a, \cdot) - \chi_\pm^{e_\pm}(\cdot, b,a) \right] \right\} \\
            &\quad + \frac{1}{(d-1)^2} \left\{ \chi_\pm^{e_\pm}(v,\chi_\pm^{e_\pm}(w,b),a)- \chi_\pm^{e_\pm}(w,\chi_\pm^{e_\pm}(v,b),a)  \right. \\
            &\qquad\qquad\qquad   + \chi_\pm^{e_\pm}(b,\chi_\pm^{e_\pm}(a,v),w)- \chi_\pm^{e_\pm}(a,\chi_\pm^{e_\pm}(b,v),w)  \\
            &\qquad\qquad\qquad - \chi_\pm^{e_\pm}(\chi_\pm^{e_\pm}(v,w),b,a) + \chi_\pm^{e_\pm}(\chi_\pm^{e_\pm}(w,v),b,a) \\
            &\qquad\qquad\qquad - \chi_\pm^{e_\pm}(\chi_\pm^{e_\pm}(b,a),v,w) + \chi_\pm^{e_\pm}(\chi_\pm^{e_\pm}(a,b),v,w)  \\
            &\qquad\qquad\qquad \left.- \trwith{E}\chi_\pm^{e_\pm}(\cdot,v,w)\chi_\pm^{e_\pm}(\cdot,b,a)\right\} \\
            &= \frac{1}{d-1} \left\{ [D^0_v \chi_\pm^{e_\pm}](w,b,a) - [D^0_w \chi_\pm^{e_\pm}](v,b,a) + [D^0_b \chi_\pm^{e_\pm}](a,v,w) - [D^0_a \chi_\pm^{e_\pm}](b,v,w) \right\} \\
            &\quad + \frac{1}{(d-1)^2} \left\{ -\scalbrack{\chi_\pm^{e_\pm}(v,a),\chi_\pm^{e_\pm}(w,b)}+ \scalbrack{\chi_\pm^{e_\pm}(w,a),\chi_\pm^{e_\pm}(v,b)} \right. \\
            &\qquad\qquad\qquad - \scalbrack{\chi_\pm^{e_\pm}(b,w),\chi_\pm^{e_\pm}(a,v)} + \scalbrack{\chi_\pm^{e_\pm}(a,w),\chi_\pm^{e_\pm}(b,v)} \\
            &\qquad\qquad\qquad \left.+ \trwith{E}\chi_\pm^{e_\pm}(\cdot,v,w)\chi_\pm^{e_\pm}(\cdot,b,a)\right\} . \\
        \end{split}
    \end{equation*}
    Finally, we calculate
    \begin{equation*}
        \begin{split}
            &-\scalbrack{\chi_\pm^{e_\pm}(v,a),\chi_\pm^{e_\pm}(w,b)}+ \scalbrack{\chi_\pm^{e_\pm}(w,a),\chi_\pm^{e_\pm}(v,b)} \\
            &- \scalbrack{\chi_\pm^{e_\pm}(b,w),\chi_\pm^{e_\pm}(a,v)} + \scalbrack{\chi_\pm^{e_\pm}(a,w),\chi_\pm^{e_\pm}(b,v)} \\
            &+ \trwith{E}\chi_\pm^{e_\pm}(\cdot,v,w)\chi_\pm^{e_\pm}(\cdot,b,a) \\
            &= \left\{ 2\scalbrack{e_\pm,e_\pm} \big[\scalbrack{w,a}\scalbrack{v,b}-\scalbrack{v,a}\scalbrack{w,b}\big] \right. \\
            &\qquad + \scalbrack{a,e_\pm} \big[\scalbrack{w,b}\scalbrack{v,e_\pm}-\scalbrack{v,b}\scalbrack{w,e_\pm}\big] \\
            &\qquad \left.+ \scalbrack{b,e_\pm} \big[\scalbrack{v,a}\scalbrack{w,e_\pm}-\scalbrack{w,a}\scalbrack{v,e_\pm}\big]\right\} .
        \end{split}
    \end{equation*}
    Employing Theorem \ref{curv_zero_dilat:thm}, which asserts that 
    \begin{equation*}
        \begin{split}
            \pm \genriem^{D^0}(a,b,v,w) &= \riem(a,b,v,w) - \frac{1}{36} H^{(2)}(a,v,b,w) \\
            &\quad - \frac{1}{36}H^{(2)}(b,v,w,a) - \frac{1}{18} H^{(2)}(v,w,a,b), 
        \end{split}
    \end{equation*}
    the result follows.
\end{proof}
\begin{theorem}\label{mixedtyperiemprop}
    The mixed-type components of the generalised Riemann curvature are given by
    \begin{equation*}
        \begin{split}
            &\pm2 \genriem^D(a,\Bar{b},v,w) \\
            &= \riem(a,\Bar{b},v,w) \mp \frac{1}{2}[\nabla_{a} H](\Bar{b},v,w) \pm \frac{1}{6}[\nabla_{\Bar{b}} H](a,v,w) \\
            &\quad - \frac{1}{12} H^{(2)}(\Bar{b},w,a,v) - \frac{1}{12} H^{(2)}(w,a,\Bar{b},v) - \frac{1}{6} H^{(2)}(a,\Bar{b},v,w) \\
            &\quad \pm \frac{1}{d-1}  [D^0_{\Bar{b}} \chi_\pm^{e_\pm}](a,v,w)   .         
        \end{split}
    \end{equation*}
    Herein, $a,v,w \in \Gamma(E_\pm)$ and $\Bar{b}\in \Gamma(E_\mp)$ arbitrary, and we denoted $d \coloneqq \dim M$.
\end{theorem} 
\begin{proof}
    We calculate
    \begin{equation*}
        \begin{split}
            D_{\Bar{b}} D_a v &= D^0_{\Bar{b}} D^0_a v + \frac{1}{d-1} D^0_{\Bar{b}}[\chi_\pm^{e_\pm}(a,v)] ,\\
            D_a D_{\Bar{b}} v &= D^0_a D^0_{\Bar{b}} v + \frac{1}{d-1} \chi_\pm^{e_\pm}(a, D^0_{\Bar{b}}v)] .
        \end{split}
    \end{equation*}
    Similarly,
    \begin{equation*}
        \begin{split}
            D_{D_{\Bar{b}} a} v &=  D^0_{D^0_{\Bar{b}} a} v + \frac{1}{d-1}\chi_\pm^{e_\pm}(D^0_{\Bar{b}} a, v), \\
            D_{D_a \Bar{b}} v &=  D^0_{D^0_a \Bar{b}} v.\\
        \end{split}
    \end{equation*}
    Put together, we obtain
    \begin{equation*}
        \begin{split}
            2\genriem^D(a,\Bar{b},v,w) -2\genriem^{D^0}(a,\Bar{b},v,w) = \frac{1}{d-1}  [D^0_{\Bar{b}} \chi_\pm^{e_\pm}](a,v,w).
        \end{split}
    \end{equation*}
    Employing Theorem \ref{curv_zero_dilat_mixed:thm},  
    the result follows.
\end{proof}
As corollaries, we obtain the pure-type components of the full generalised Ricci tensor and then the generalised scalar curvature. The full generalised Ricci tensor is defined as the trace (taken with the inner product $\scalprodmap$)
\begin{equation*}
    \genfullric^D(a,b) \coloneqq \trwith{E} \genriem^D(\cdot,a,\cdot,b), \qquad\quad a,b \in \Gamma(E).
\end{equation*}
Given a generalised metric $\genmet$ and a divergence operator $\divergence$, we remark that the mixed-type components of the generalised Ricci curvature are the same for any choice of generalised Levi-Civita connection with divergence $\divergence$ \cite[Proposition 4.5]{streets2024genricciflow}. We denote them by
\begin{equation}\label{genric_def_eq}
    \genric^\pm(\genmet,\divergence) \coloneqq \restr{\genric^D}{E_\mp \times E_\pm} .
\end{equation}
The sum $\genric^+(\genmet,\divergence)+ \genric^-(\genmet,\divergence)\in \Gamma (E_+\vee E_-)$ is what we call the generalised 
Ricci tensor. Note that there are different but related definitions of generalised Ricci curvature in the literature. They have been compared in \cite{rubiogenricequivalence}. The comparison has been summarised in \cite[Remark 5.13]{vicenteoskar}.

Finally, from the pure-type components of $\genfullric^D$, we define the generalised scalar curvature
\begin{equation*}
    \genscal(\genmet,\divergence) \coloneqq \frac{1}{2}\trwith{\genmet}\genfullric^D.
\end{equation*}
We remark that (as for the generalised Ricci tensor) the generalised scalar curvature is the same for any generalised Levi-Civita connection with divergence $\divergence$ \cite[Lemma 4.11]{streets2024genricciflow}.
\begin{corollary}\label{genric_puretype_cor}
    The pure-type part of the generalised Ricci curvature is given by
    \begin{equation*}
        \begin{split}
            \restr{\genfullric^D(a,b)}{E_\pm \times E_\pm} &= \ric   - \frac{1}{12}H^{2}\pm \frac{\divergence_g(e_\pm)}{d-1}\:g \pm\frac{ d-2}{d-1} [g\nabla e_\pm]^\sym \\
            &\quad + \frac{3-2d}{2(d-1)^2} \absolute{e_\pm}^2_\genmet \: g +\frac{d-2}{2(d-1)^2}g\pi e_\pm \otimes g\pi e_\pm . \\
        \end{split}
    \end{equation*}
\end{corollary}
\begin{proof}
    This is a straightforward computation. Let $a,b \in \Gamma(E_\pm)$.
    \begin{equation*}
        \begin{split}
            &\genfullric^D(a,b) \\
            &= \trwith{E} \genriem^D(\cdot,a,\cdot,b) \\
            &= \trwith{E_\pm} \genriem^D(\cdot,a,\cdot,b) \\
            &= \trwith{g} \left\{\riem(\cdot,a,\cdot,b)  - \frac{1}{36} H^{(2)}(\cdot,\cdot,a,b) - \frac{1}{36}H^{(2)}(a,\cdot,b,\cdot) - \frac{1}{18} H^{(2)}(\cdot,b,\cdot,a) \right. \\
            &\quad \pm \frac{1}{2(d-1)} \left\{ [D^0_{(\cdot)} \chi_\pm^{e_\pm}](b,a,\cdot) - [D^0_b \chi_\pm^{e_\pm}](\cdot,a,\cdot)\right. \\
            &\qquad\qquad\qquad \left. + [D^0_a \chi_\pm^{e_\pm}](\cdot,\cdot,b) - [D^0_{(\cdot)} \chi_\pm^{e_\pm}](a,\cdot,b) \right\} \\
            &\quad + \frac{1}{2(d-1)^2} \left\{2 \genmet(e_\pm,e_\pm) \big[\genmet(b,\cdot)\genmet(\cdot,a)-\genmet(\cdot,\cdot)\genmet(b,a)\big] \right. \\
            &\qquad\qquad\qquad + \genmet(\cdot,e_\pm) \big[\genmet(b,a)\genmet(\cdot,e_\pm)-\genmet(\cdot,a)\genmet(b,e_\pm)\big] \\
            &\qquad\qquad\qquad + \left.\left.\genmet(a,e_\pm) \big[\genmet(\cdot,\cdot)\genmet(b,e_\pm)-\genmet(b,\cdot)\genmet(\cdot,e_\pm)\big]\right\}\right\} \\
            &\stackrel{*}{=} \ric(a,b)   - \frac{1}{12}H^{2}(a,b) \\
            &\quad\pm \frac{1}{2(d-1)} \left\{g(b,a)\divergence_g(e_\pm) - g(D^0_b e_\pm,a) +(d-1) g(D^0_b e_\pm,a)\right. \\
            &\qquad\qquad\qquad \left. + (d-1) g(D^0_a e_\pm,b)  +g(a,b)\divergence_g (e_\pm) - g(D^0_a e_\pm,b)\right\} \\
            &\quad + \frac{1}{2(d-1)^2} \left\{2 \absolute{e_\pm}^2_\genmet \big[\genmet(b,a)-d\, \genmet(b,a)\big] + \genmet(b,a)\absolute{e_\pm}_\genmet^2-\genmet(e_\pm,a)\genmet(b,e_\pm)\big] \right.\\
            &\qquad\qquad\qquad + \left.\left.\genmet(a,e_\pm) \big[d\, \genmet(b,e_\pm)-\genmet(b,e_\pm)\big]\right\}\right\} \\
            &\stackrel{\#}{=} \ric(a,b)   - \frac{1}{12}H^{2}(a,b)\pm \frac{g(a,b)}{d-1}\divergence_g(e_\pm) \pm\frac{ d-2}{d-1} [g\nabla e_\pm]^\sym(a,b) \\
            &\quad + \frac{1}{2(d-1)^2} \left\{(3-2d) \absolute{e_\pm}^2_\genmet g(a,b) +(d-2)g(e_\pm,a)g(b,e_\pm)\right\} .\\
        \end{split}
    \end{equation*}
    In $*$, we used that
    \begin{equation*}
        [D^0_u\chi_\pm^{e_\pm}](v,w,x) = g(v,w)g(D^0_u e_\pm,x) - g(v,x)g(D^0_u e_\pm,w),
    \end{equation*}
    for $u,v,w,x \in \Gamma(E_\pm)$. In $\#$, we used that 
    \[ g(D^0_ae_\pm,b) + g(D^0_ae_\pm,b) = 2 g(\nabla e_\pm)^{\mathrm{sym}}(a,b),\]
    see Proposition \ref{D_can_coeff:prop}.
\end{proof}
\begin{corollary}\label{traG:cor}
    The generalised scalar curvature is given by
    \begin{equation}\label{genscaleq}
        \genscal(\genmet, \divergence) = \rscal - \frac{\absolute{H}^2}{12} + \divergence^\genmet(e_+ - e_-) - \frac{1}{2}\absolute{e}_\genmet^2.
    \end{equation}
    Herein, $\divergence = \divergence^\genmet - \scalbrack{e, \cdot}$.
\end{corollary}
\begin{proof}
    We compute with Corollary \ref{genric_puretype_cor}
    \begin{equation*}
        \begin{split}
            \trwith{\genmet}\restr{\genfullric^D}{E_\pm \times E_\pm} &= \rscal - \frac{1}{12}\absolute{H}^2\pm \frac{d}{d-1}\divergence_g(e_\pm) \pm \frac{ d-2}{d-1} \divergence_g(e_\pm) \\
            &\quad + \frac{1}{2(d-1)^2} \left\{(3-2d) \absolute{e_\pm}^2_\genmet d +(d-2)\absolute{e_\pm}^2_\genmet\right\} \\
            &= \rscal - \frac{1}{12}\absolute{H}^2\pm 2 \divergence_g(e_\pm) - \absolute{e_\pm}^2_\genmet .
        \end{split}
    \end{equation*}
    Thus,
    \begin{equation*}
        \begin{split}
            \genscal &= \frac{1}{2}\left\{\trwith{\genmet}\restr{\genfullric^D}{E_+ \times E_+} + \trwith{\genmet}\restr{\genfullric^D}{E_- \times E_-}\right\} \\
            &= \rscal - \frac{1}{12}\absolute{H}^2 +\divergence_g(e_+ - e_-) - \frac{1}{2}\absolute{e}^2_\genmet,
        \end{split}
    \end{equation*}
    as claimed.
\end{proof}
\begin{corollary} \label{trE:cor} The trace of the full generalised Ricci tensor with respect to the scalar product $\langle \cdot, \cdot \rangle$ is given by:
\[ \mathrm{tr}_E\,\genfullric^D=2\divergence_g(e) - (|e_+|^2_\genmet -|e_-|^2_\genmet). \]
\end{corollary}
\begin{proof} This follows from the formulas in the previous proof taking into account that $\mathcal{G}$ and $\langle \cdot ,\cdot\rangle$ differ by a minus sign on $E_-$.  
\end{proof}
\begin{proposition} \label{Kretsch:prop}The following generalised Kretschmann scalar is a (non-trivial) geometric invariant of the pair $(\mathcal G, \mathrm{div})$:
\[ |\genriem^D|^2_{\mathcal{G}} =\mathcal{G}^{II'}\mathcal{G}^{JJ'}\mathcal{G}^{KK'}\mathcal{G}^{LL'}\genriem^D_{IJKL}\genriem^D_{I'J'K'L'}, \]
where the sum is over all components. 
    \end{proposition} 
    \begin{proof}
       The invariance follows from the fact that the generalised connection $D=D^{\mathcal G, \mathrm{div}}$ is 
       canonically associated with the pair $(\mathcal G, \mathrm{div})$. Inspection of the formulas for the 
       generalised Riemann tensor in the case $H=0$, $\mathrm{div}=\mathrm{div}^{\mathcal G}$ show that 
        \[ \left. \left( |\genriem^D|^2_{\mathcal{G}}\right)\right|_{H=0,e=0} = 4|\riem|_g^2,\]
        and therefore the non-triviality of the invariant. 
    \end{proof} 
\begin{corollary}\label{genriccor}
    The mixed-type components of the generalised Ricci curvature are given by
    \begin{equation}\label{genriceq}
        4\restr{\genric(\genmet,\divergence)}{E_\mp \times E_\pm}= 4\ric - H^{2} \mp 2\extd^* H +4 [\nabla\xi]^{\mathrm{sym}} \pm 4[\nabla g X]^{\mathrm{antisym}} \mp 2 H(\xi) .
    \end{equation}
    Herein, $\divergence = \divergence^\genmet -2 \scalbrack{X+\xi, \cdot}$ and the $(0,2)$ tensor $H^2$ is defined by $H^2(X,Y) = \trwith{g} H^{(2)}(X, \cdot, Y, \cdot)$.
\end{corollary}
\begin{proof}
    We first compute the traces over the subspaces $E_+$ and $E_-$. Let $a \in \Gamma(E_\pm)$ and $\Bar{b} \in \Gamma(E_\mp)$. Then
    \begin{equation*}
        \begin{split}
            &2\trwith{E_\pm} \genriem^D(\cdot,\Bar{b},\cdot, a)\\
            &= \trwith{g} \left\{\riem(\cdot,\Bar{b},\cdot,a) \mp \frac{1}{2}[\nabla_{(\cdot)} H](\Bar{b},\cdot,a) \pm \frac{1}{6}[\nabla_{\Bar{b}} H](\cdot,\cdot,a) \right.\\
            &\quad - \frac{1}{12} H^{(2)}(\Bar{b},a,\cdot,\cdot) - \frac{1}{12} H^{(2)}(a,\cdot,\Bar{b},\cdot) - \frac{1}{6} H^{(2)}(\cdot,\Bar{b},\cdot,a) \\
            &\quad \left. \pm \frac{1}{d-1}  [D^0_{\Bar{b}} \chi_\pm^{e_\pm}](\cdot,\cdot,a) \right\} \\
            &= \ric(\Bar{b},a) \mp \frac{1}{2} \extd^*H(\Bar{b},a) - \frac{1}{4} H^2(\Bar{b},a) \pm g(D^0_{\Bar{b}} {e_\pm},a) \\
            &= \ric(\Bar{b},a) \mp \frac{1}{2} \extd^*H(\Bar{b},a) - \frac{1}{4} H^2(\Bar{b},a)  \\
            &\quad\pm  g(\nabla_{\Bar{b}} X,a) + [\nabla_{\Bar{b}} \xi](a) + \frac{1}{2} H(\Bar{b},X,a) \pm \frac{1}{2} H(\Bar{b},\xi,a), \\
        \end{split}
    \end{equation*}
    where in the last line we used that
    \begin{equation*}
        \begin{split}
            \pi D^0_{\Bar{b}} e_\pm  &= \nabla^\pm_{\Bar{b}} \pi e_\pm = \nabla^\pm_{\Bar{b}} (X \pm g^{-1}\xi) \\
            &=\nabla_{\Bar{b}} X \pm g^{-1}\left[\nabla_{\Bar{b}} \xi \pm \frac{1}{2} \left(H(\Bar{b},X) \pm  H(\Bar{b},g^{-1}\xi)\right)\right].
        \end{split}
    \end{equation*}
    Interchanging \enquote{$+$} with \enquote{$-$} and $a$ with $\Bar{b}$, we obtain that
    \begin{equation*}
        \begin{split}
            &2\trwith{E_\mp} \genriem^D(\cdot,a,\cdot, \Bar{b}) \\
            &= \ric(a,\Bar{b}) \pm \frac{1}{2} \extd^*H(a,\Bar{b}) - \frac{1}{4} H^2(a,\Bar{b})\\
            &\quad \mp  g(\nabla_{a} X,\Bar{b}) + [\nabla_{a} \xi](\Bar{b})  + \frac{1}{2} H(a,X,\Bar{b}) \mp \frac{1}{2} H(a,\xi,\Bar{b}) .\\
        \end{split}
    \end{equation*}
    Finally, we compute
    \begin{equation*}
        \begin{split}
            &4\:\genric(\genmet,\divergence)(\Bar{b},a) = 4\:\trwith{E} \genriem^D(\cdot,\Bar{b},\cdot, a) \\
            &= 4\:\trwith{E_\pm} \genriem^D(\cdot,\Bar{b},\cdot, a) + 4\:\trwith{E_\mp} \genriem^D(\cdot,\Bar{b},\cdot, a) \\
            &= 4\:\ric(\Bar{b},a) \mp 2\: \extd^*H(\Bar{b},a) - H^2(\Bar{b},a)  \\
            &\quad\pm  4 [\nabla gX]^\antisym(\Bar{b},a) + 4[\nabla \xi]^\sym(\Bar{b},a) \mp 2 H(\xi,\Bar{b},a). \\
        \end{split}
    \end{equation*}
\end{proof}

\section{Comparison with the physics literature}
\label{Comparison:sec}
In this section we specialise our formulas to two cases which have been considered in the physics literature: 
\begin{enumerate}
    \item the standard setting in string theory and supergravity when $e\in \Gamma (E)\cong \Gamma (\mathbb{T}M)$ 
is an exact one-form  (up to a conventional factor the differential of the dilaton function) and 
\item the vector-deformed setting \cite{Arutyunov:2015mqj}. 
\end{enumerate}
\paragraph{Expressing generalised Ricci and scalar curvature in index notation.}
Here we spell out the above formulas for the generalised Ricci and scalar curvature in terms 
of the vector and co-vector components of $e=2(X+\xi)$. The projections $e_\pm$ can be 
written as 
\[ e_\pm = X_\pm \pm gX_\pm,\quad X_\pm \in \Gamma (TM). \]
Solving for $X$ and $\xi$ we have 
\begin{eqnarray*}
    2X &=& X_+ + X_-,\\
    2 \xi &=& g(X_+-X_-).
\end{eqnarray*}
The section $e_+-e_-\in \Gamma (E)$ can be written in terms of $X$ and $\xi$ as
\[ 
e_+-e_- = X_+-X_- +g(X_++X_-)= 2(g^{-1}\xi +gX). 
\] 
This implies that 
\[ \mathrm{div}^\mathcal{G}(e_+-e_-)= 2\mathrm{div}^g(g^{-1}\xi) = 2 \nabla_\mu \xi^\mu. \]
The square-norm $|e|_{\mathcal G}^2$ is given in terms of $X$ and $\xi$ by
\[ |e|_{\mathcal G}^2 = |e_+|_{\mathcal G}^2 + |e_-|_{\mathcal G}^2= |X_+|^2_g + |X_-|^2_g = 
2(|X|^2_g+|\xi|_g^2),\]
where we have used that 
\begin{eqnarray*}|2X|^2_g &=& |X_++X_-|^2_g = |X_+|^2_g+|X_-|_g^2 + 2g(X_+,X_-),\\
|2\xi|^2_g &=& |X_+-X_-|^2_g = |X_+|^2_g+|X_-|_g^2 - 2g(X_+,X_-). 
\end{eqnarray*}
Similarly, 
\[ -\frac12 (|e_+|_{\mathcal G}^2 - |e_-|_{\mathcal G}^2)=
-\frac12 (|X_+|^2_g-|X_-|_g^2) = -2\xi (X),\]
since $X_\pm = X \pm g^{-1}\xi$. Now we can write the two generalised scalar curvatures in index notation 
\begin{eqnarray*} 
          \genscal(\genmet, \divergence) = \frac12 \mathrm{tr}_{\mathcal G}\,\genfullric^D &=&  \rscal - \frac{1}{12}H_{\mu \nu \rho}H^{\mu \nu \rho} + 2 \nabla_\mu \xi^\mu - X^\mu X_\mu -\xi^\mu \xi_\mu,\\
        \frac12 \mathrm{tr}_E\,\genfullric^D &=& 2\divergence_g(X) - 2\xi (X) = 2(\nabla_\mu X^\mu - \xi_\mu X^\mu).
    \end{eqnarray*}
    Next we write the generalised Ricci tensor in index notation:
    \begin{eqnarray*} 
        4\genric(\genmet,\divergence)^{\pm}_{\mu \nu} &=& 4\ric_{\mu \nu} - H_{\mu \rho \sigma}H_\nu^{\;\; \rho \sigma}
        +4 \nabla_{(\mu} \xi_{\nu )}\\   
        &&\pm 2\nabla^\rho H_{\rho \mu \nu}  \pm 4\nabla_{[ \mu} X_{\nu ]} \mp 2 H_{\rho \mu \nu }\xi^\rho, 
    \end{eqnarray*}
    where we have separated the symmetric and the skew-symmetric parts.
    \paragraph{Specialisation to the standard string theory setting.}
    Specialising to $\xi_\mu = 2\partial_\mu \varphi$ and $X^\mu=0$ we recover the formulas  
    \[ \genscal(\genmet, \divergence) = \rscal - \frac{1}{12}H_{\mu \nu \rho}H^{\mu \nu \rho} + 4 \nabla_\mu \nabla^\mu \varphi  -4\partial_\mu \varphi \partial^\mu \varphi\]
    and 
    \begin{eqnarray*} 
        \genric(\genmet,\divergence)^{\pm}_{\mu \nu} &=& \ric_{\mu \nu} - \frac14 H_{\mu \rho \sigma}H_\nu^{\;\; \rho \sigma}
        +2 \nabla_{\mu} \partial_{\nu} \varphi    
        \pm \frac12\nabla^\rho H_{\rho \mu \nu}  \mp  H_{\rho \mu \nu }\partial^\rho \varphi\\  
        &=& \ric_{\mu \nu} - \frac14 H_{\mu \rho \sigma}H_\nu^{\;\; \rho \sigma}
        +2 \nabla_{\mu} \partial_{\nu} \varphi    
        \pm \frac12 e^{2 \varphi} \nabla^\rho (e^{-2\varphi} H_{\rho \mu \nu}), 
    \end{eqnarray*}
     in concordance with \cite{Coimbra:2011nw}. The field equations of the NS-NS sector of type-IIA/B supergravity, and, equivalently, 
     the equations expressing the vanishing of the beta-functions for the NS-NS sector of critical type-IIA/B superstring theory to leading order in $\alpha'$ are the generalised Einstein equations $\genric(\genmet,\divergence)^{\pm}=0$
     and $\genscal(\genmet, \divergence)=0$.  We note that the condition $X=0$ imposed on $(X,\xi)$ implies that $\mathrm{tr}_E\,\genfullric^D=0$.
\paragraph{Specialisation to vector-deformed supergravity.}
     Next we compare with the corresponding formulas in vector-deformed supergravity, where the conditions imposed on the pair $(X,\xi)$ are relaxed with respect to the standard case, as follows:
     firstly, it is assumed that the pair $(X,\xi)$ is compatible, i.e.\ $X$ is Killing and $d\xi = H(X)$. Secondly, one imposes that 
     $\xi (X)=0$.\footnote{For completeness we mention that in \cite{Arutyunov:2015mqj} the further assumption is made that 
     there exists a two-form $B$, with $H=dB$, which is invariant under the action generated by $X$, $\mathcal{L}_XB=0$. Then the one-form $\xi$ is a linear combination of the one-form $BX$ and
     an exact one-form, which is interpreted as the differential $d\varphi$ of the dilaton. These assumptions do not play a role in the context of our paper, since all our expressions are given in terms of the closed three-form $H$.} We first rewrite the generalised scalar curvature by expressing
     $2 \nabla_\mu \xi^\mu - X^\mu X_\mu -\xi^\mu \xi_\mu$ in terms of $\tilde{X}^\mu=\frac12 (-X^\mu +\xi^\mu )$:
     \[ 2\nabla_\mu \xi^\mu - X^\mu X_\mu -\xi^\mu \xi_\mu = 2 \nabla_\mu (-X^\mu + \xi^\mu)-(-X+\xi)^\mu (-X+\xi)_\mu
     =4 \nabla_\mu\tilde{X}^\mu -4 \tilde{X}^\mu\tilde{X}_\mu.\]
     As a consequence, 
     \[ \genscal(\genmet, \divergence) = \rscal - \frac{1}{12}H_{\mu \nu \rho}H^{\mu \nu \rho} + 4 \nabla_\mu\tilde{X}^\mu -4 \tilde{X}^\mu\tilde{X}_\mu .\]
     Comparing with \cite{Arutyunov:2015mqj} we see that $\genscal(\genmet, \divergence)=0$ is one of the field equations 
     of vector-deformed supergravity. Now we rewrite the symmetric part of $\genric(\genmet,\divergence)^{\pm}_{\mu \nu}$ as 
     \begin{eqnarray*} \genric(\genmet,\divergence)^{\pm}_{(\mu \nu )} &=& \ric_{\mu \nu} - \frac14 H_{\mu \rho \sigma}H_\nu^{\;\; \rho \sigma}
        + \nabla_{(\mu} \xi_{\nu )} = \ric_{\mu \nu} - \frac14 H_{\mu \rho \sigma}H_\nu^{\;\; \rho \sigma}
        + 2\nabla_{(\mu} \tilde{X}_{\nu )}
     \end{eqnarray*}
     and the skew-symmetric part as 
     \begin{eqnarray*} \genric(\genmet,\divergence)^{\pm}_{[\mu \nu ]} &=&\pm \frac12 \nabla^\rho H_{\rho \mu \nu}  \pm \nabla_{[ \mu} X_{\nu ]} \mp \frac12  H_{\rho \mu \nu }\xi^\rho\\
     &=& \pm \frac12 \nabla^\rho H_{\rho \mu \nu} \mp 2 \nabla_{[\mu} \tilde{X}_{\nu]} \mp H_{\rho \mu \nu}\tilde{X}^\rho.\end{eqnarray*}
     Comparing with \cite{Arutyunov:2015mqj} we see that $\genric(\genmet,\divergence)^{\pm}=0$ are the remaining two field equations of vector-deformed supergravity. 
     We note that $\mathrm{tr}_E\,\genfullric^D =0$ under the assumptions underlying vector-deformed supergravity. 
     In fact, $X$ is Killing and therefore divergence-free, and $\xi (X)=0$ is another assumption. 

\cleardoublepage

\bibliographystyle{unsrturl}  
\bibliography{literatur}

V.\ Cort\'es: vicente.cortes@uni-hamburg.de

Department of Mathematics, University of Hamburg,  Bundesstr.\ 55, 20146  Hamburg, Germany.

M.\ Mackevicius: Matas.Mackevicius@liverpool.ac.uk

Department of Mathematical Sciences, University of Liverpool, Peach Street, Liverpool L69 7ZL, UK

T.\ Mohaupt: Thomas.Mohaupt@liverpool.ac.uk

Department of Mathematical Sciences, University of Liverpool, Peach Street, Liverpool L69 7ZL, UK

O.\ Schiller: 
oskar.schiller@uni-hamburg.de

Department of Mathematics, University of Hamburg,  Bundesstr.\ 55, 20146  Hamburg, Germany.
\cleardoublepage

\end{document}